\documentclass[11pt]{amsart}
\usepackage{amssymb}
\usepackage[margin=1in]{geometry}
\usepackage[colorlinks]{hyperref}
\newtheorem{theorem}{Theorem}[section]

\newtheorem{cor}[theorem]{Corollary}

\theoremstyle{definition}

\theoremstyle{remark}

\numberwithin{equation}{section}
\begin{document}

\newcommand{\spacing}[1]{\renewcommand{\baselinestretch}{#1}\large\normalsize}
\spacing{1.14}

\title{Non-Berwaldian Randers metrics of Douglas type on four-dimensional hypercomplex Lie groups}

\author {M. Hosseini}

\address{Department of Mathematics\\ Faculty of  Sciences\\ University of Isfahan\\ Isfahan\\ 81746-73441-Iran.} \email{hoseini\_masomeh@ymail.com}

\author {H. R. Salimi Moghaddam}

\address{Department of Mathematics\\ Faculty of  Sciences\\ University of Isfahan\\ Isfahan\\ 81746-73441-Iran.} \email{hr.salimi@sci.ui.ac.ir and salimi.moghaddam@gmail.com}

\keywords{hypercomplex structure, Randers metric, Finsler metric of Berwald type, Finsler metric of Douglas type, flag curvature\\
AMS 2010 Mathematics Subject Classification: 53C26, 53C60, 22E60.}


\begin{abstract}
In this paper we classify all non-Berwaldian Randers metrics of Douglas type arising from invariant hyper-Hermitian metrics on simply connected four-dimensional real Lie groups. Also, the formulas of the flag curvature are given and it is shown that, in some directions, the flag curvature of the Randers metrics and the sectional curvature of the hyper-Hermitian metrics have the same sign.
\end{abstract}

\maketitle

\section{\textbf{Introduction}}
In this paper our aim is to characterize non-Berwaldian Randers metrics of Douglas type arising from invariant hyper-Hermitian metrics on simply connected four-dimensional real Lie groups and study their flag curvatures. Hyper-Hermitian manifolds are very interesting manifolds in mathematics which have been studied by many mathematicians. In fact the main object of HKT-geometry (Hyper-K\"{a}hler geometry with torsion) is hyper-Hermitian manifold (see \cite{Grantcharov, Joyce, Poon}). On the other hand, these spaces have found many applications in physics (see \cite{Alvarez-Freedman, Gauduchon-Tod, Gutowski-Sabra, Merkulov-Pedersen-Swann}). For example they are the target spaces for $\sigma$-models with $(4,0)$-supersymmetry (see \cite{Alvarez-Freedman}).\\
Among geometric structures, invariant structures on Lie groups and homogeneous spaces are interesting topics that deserve more study. There are many results concerning invariant hypercomplex structures on Lie groups. For example Joyce has studied the existence of invariant hypercomplex structures on compact Lie groups in \cite{Joyce}. In \cite{Dotti-Fino}, Dotti and Fino have categorized $8$-dimensional simply connected nilpotent Lie groups equipped with a left invariant hypercomplex structure. A more important result, for this paper, is the classification of invariant hypercomplex structures on 4-dimensional simply connected Lie groups given by Barberis (see \cite{Barberis}). The Levi-Civita connection and sectional curvature of these spaces have been computed by the second author of this work in \cite{Salimi1}.\\
On the other hand, using a left invariant Riemannian metric and a left invariant vector field ($1$-form), we can construct a left invariant Randers metric, a special type of Finsler metrics, on Lie groups (see \cite{Deng-Book}). Randers metrics, in general case, have been introduced by G. Randers in \cite{Randers} because of its applications in general relativity. The second author, in his previous article \cite{Salimi2}, has studied Randers metrics of Berwald type, arising from invariant hyper-Hermitian metrics, on four-dimensional simply connected Lie groups. In this work we study, the more general case, Randers metrics of Douglas type on such spaces and give explicit formulas of their flag curvatures.


\section{\textbf{Preliminaries}}
A hypercomplex structure, on a $4n$-dimensional manifold $M$, is a family ${\mathcal{H}}= \{J_i\}_{i=1,2,3}$ consisting three complex structures $J_1$, $J_2$ and $J_3$ such that they satisfy the following equations:
\begin{equation*}
J_1J_2=J_3=-J_2J_1.
\end{equation*}
In fact a hypercomplex structure consists of three fiberwise endomorphism $J_1$, $J_2$ and $J_3$ of $TM$ such that
\begin{eqnarray}
  && J_1J_2=J_3=-J_2J_1, \\
  && J_i^2=-I_{TM}, i=1,2,3,\\
  && N_i=0, i=1,2,3,
\end{eqnarray}
where $N_i$ denotes the Nijenhuis tensor of $J_i$ which is defined as follow
\begin{equation}\label{Nijenhuis tensor}
N_i(X,Y) = [X,Y]+J_i[J_iX,Y]+J_i[X,J_iY]-[J_iX,J_iY],
\end{equation}
for all vector fields $X$, $Y$ on $M$.\\
Let $g$ be a Riemannian metric on a hypercomplex manifold $(M,{\mathcal{H}})$. Then $g$ is called hyper-Hermitian if
\begin{equation*}
g(X,Y)=g(J_iX,J_iY), i=1,2,3,
\end{equation*}
for all vector fields $X$, $Y$ on $M$.\\
A hypercomplex structure ${\mathcal{H}}= \{J_i\}_{i=1,2,3}$ on a Lie group $G$ is called left invariant if, for any $a\in G$, commutes with left translation $l_a$ (see \cite{Barberis} and \cite{Salimi2}), it means that for any $a \in G$,
\begin{equation*}
J_i\circ Tl_a = Tl_a\circ J_i,
\end{equation*}
where $Tl_a$ denotes the tangent map of $l_a$.\\

Now we give some basic definitions of Finsler geometry which will be used throughout this article.\\
Let $M$ be a smooth manifold and $TM$ be its tangent bundle. By definition, a Finsler metric is a function $F:TM \longrightarrow \left[ 0, \infty \right)$ with the following properties:
\begin{itemize}
\item[(i)] $F$ is a smooth function on $TM\backslash \{0\}$,
\item[(ii)] $F(x,\lambda y)=\lambda F(x,y)$, \ \ \ \ for all $\quad  \lambda>0$,
\item[(iii)] The hessian matrix $(g_{ij}(x,y))=\left( \dfrac{1}{2} \dfrac{\partial ^2 F^2}{\partial y^i \partial y^j}\right) $ is positive definite for all $(x,y) \in TM\backslash \{0\}$.
\end{itemize}
An important family of Finsler metrics is the family of Randers metrics which arises from Riemannian metrics and 1-forms on manifolds (see \cite{Randers}). In other words, Randers metrics are as follows
\begin{equation}\label{Randers metric}
F(x,y) = \sqrt{g(x) (y,y)} + b(x) (y),
\end{equation}
where $g = g_{ij}dx^i \otimes dx^j$ is a Riemannian metric and $b = b_i dx^i$ is a 1-form on $M$.\\
The necessary and sufficient condition, to $F$ be a Finsler metric, is that $\Vert b\Vert_g =\sqrt{g^{ij} (x)b_i(x) b_j(x) } < 1$, for any $x \in M$, where $(g^{ij}(x))$ is the inverse matrix of $(g_{ij}(x))$ (for more details see  \cite{Chern-Shen}).
Easily we can rewrite any Randers metric $F(x,y) =  \sqrt{g(x) (y,y)} + b(x) (y)$ as follows
\begin{equation}\label{Randers metric}
F(x,y) = \sqrt{g(x) (y,y)} + g(x) ( Q(x), y ),
\end{equation}
where $Q$ denotes the dual vector field of $b$ (see \cite{Deng-Book} and \cite{Salimi2}).\\
Let $(M,F)$ be a Finsler manifold. In a standard local coordinate system $(x^i,y^i)$ in $TM$, the spray coefficients $G_i$ of $F$ are defined by
\begin{equation}\label{spray coefficients}
G^i:=\frac{1}{4}g^{il}([F^2]_{x^my^l}y^m-[F^2]_{x^l}).
\end{equation}
A Finsler metric is called a Douglas metric if there exists a positively $y$-homogeneous function of degree one $P(x ,y)$ such that the spray coefficients $G^i$ satisfy the following equations:
\begin{equation}\label{spray coefficients of Douglas metrics}
    G^i=\frac{1}{2}\Gamma^i_{jk}(x)y^jy^k+P(x,y)y^i.
\end{equation}
A special case happens when $P(x ,y)=0$. Then the spray coefficients $G^i=\frac{1}{2}\Gamma^i_{jk}(x)y^jy^k$ are quadratic in $y$. In this case, $F$ is called a Berwald metric (see \cite{Bacso-Matsumoto} and \cite{Chern-Shen}). \\
Similar to the Riemannian case, a Finsler metric $F$ on a Lie group $G$ is said to be left invariant if
\begin{equation}
F(x,y)=F(e,T_x\,l\,_{x^{-1}}y), \ \ \ \forall x\in G, y\in T_x G,
\end{equation}
where $e$ denotes the unit element of $G$.\\
Fix a left invariant Riemannian metric $g$ on a Lie group $G$ and consider a left invariant vector field $Q$ such that $g(Q,Q) < 1$, then we have a left invariant Randers metric on $G$, using the formula \ref{Randers metric}.\\
Flag curvature which is a generalization of the concept of sectional curvature to Finsler geometry
is an important quantity in Finsler geometry. It is computed by the following formula:
\begin{equation}\label{flag curvature main formula}
K(P,V)=\dfrac{g_V \left( R_V \left( U\right) ,U\right) }{g_V (V,V) g_V (U,U)-g_V^2 (V,U)},
\end{equation}
where $P=\textit{span} \lbrace U,V \rbrace $, $g_V(U,W)=\frac{1}{2}\frac{\partial ^2}{\partial s \partial t}F^2(V+sU+tW)\mid _{s=t=0}$, $R_V(U)=R(U,V)V=\nabla _U \nabla _V V -\nabla _V \nabla _U V- \nabla _{[U,V]}V$ and $\nabla$ is the Chern connection of $F$ (for more details see \cite{Bao-Chern-Shen,Chern-Shen}).\\


\section{Randers metrics of Douglas type and their flag curvature}
In \cite{Salimi2}, Randers metrics of Berwald type, arising from invariant hyper-Hermitian metrics, on four-dimensional simply connected Lie groups have been studied. The main tool which is used for this characterization is the classification of four-dimensional hypercomplex Lie groups given by Barberis in \cite{Barberis}. Let $G$ be a simply connected four-dimensional real Lie group equipped with a left invariant hyper-Hermitian metric, and $\frak{g}$ denotes its Lie algebra. She has shown  that, if $\frak{g}$ isn't commutative, then it is isomorphic to one of the following Lie algebras:
\begin{itemize}
\item[(1)] $[Y,Z]=W,\, [Z,W]=Y,\, [W,Y]=Z,\, X$ central,
\item[(2)] $[X,Z]=X,\, [Y,Z]=Y,\, [X,W]=Y,\, [Y,W]=-X$,
\item[(3)] $[X,Y]=Y,\, [X,Z]=Z,\, [X,W]=W$,
\item[(4)] $[X,Y]=Y,\, [X,Z]=\dfrac{1}{2}Z,\, [X,W]=\dfrac{1}{2}W,\, [Z,W]=\dfrac{1}{2}Y$,
\end{itemize}
where $\{ X, Y, Z, W\}$ is an orthonormal basis for $\mathfrak{g}$ (for more details see \cite{Barberis}).\\

Now using the above result we give the following theorem which classify Randers metrics of Douglas type on such spaces.

\begin{theorem}\label{classification theorem}
Let $G$ be a non-commutative four-dimensional hypercomplex simply connected Lie group equipped with a left invariant hyper-Hermitian metric $g$. Suppose that $F$ is a non-Berwaldian Randers metric of Douglas type defined by $g$ and a left invariant vector field $Q$. Then the Lie algebra $\frak{g}$ of $G$ and the vector field $Q$ are as follows:
\begin{itemize}
  \item[(i)] $\mathfrak{g}$ belongs to case (2), $Q=pZ+qW$, $\sqrt{p^2+q^2}<1$ and $p\neq0$,
  \item[(ii)] $\mathfrak{g}$ belongs to case (3), $Q=qX$, $|q|<1$ and $q\neq0$,
  \item[(iii)] $\mathfrak{g}$ belongs to case (4), $Q=qX$, $|q|<1$ and $q\neq0$,
\end{itemize}
where $p, q\in \Bbb{R}$ and $\{ X, Y, Z, W\}$ is an orthonormal basis for $\mathfrak{g}$ with respect to $g$.
\end{theorem}
\begin{proof}
We know that $F$ is of Douglas type if and only if $Q$ is $g$-orthogonal to $[\frak{g},\frak{g}]$ (see \cite{An-Deng Monatsh}). Now it is sufficient to check, the classification given in \cite{Barberis}, case by case and use the results of \cite{Salimi2}.\\
Let $(G,g)$ belong to the spaces of case (1) of the classification given by Barberis, then easily we can see $Q$ is orthogonal to $[\frak{g},\frak{g}]$ if and only if $Q=qX$. But, in \cite{Salimi2}, it is shown that in this case $F$ is of Berwald type. So in case (1) we have not any non-Berwaldian Randers metric of Douglas type. For the case (2), $Q$ is perpendicular to the derived algebra if and only if $Q=pZ+qW$, for some $p,q \in \mathbb{R}$. Now it is sufficient to let $p\neq0$, because if $p=0$ then $F$ is of Berwald type by \cite{Salimi2}. \\
Similarly, in cases (3) and (4), the only left invariant vector fields which are perpendicular to the ideal $[\mathfrak{g} , \mathfrak{g} ]$ are of the form $Q=qX$ for some $q \in \mathbb{R}$. In \cite{Salimi2}, it has been proven that the cases (3) and (4) do not admit any Berwaldian Randers metric arising from left invariant hyper-Hermitian metric $g$. So any Randers metrics defined by $g$ and $Q$ are non-Berwaldian Randers metrics of Douglas type.
\end{proof}

Therefore, using the above theorem and \cite{Salimi2}, we have the following corollary.
\begin{cor}
Lie groups of case (1) admits only Berwaldian Randers metrics of Douglas type arising from left invariant hyper-Hermitian metrics $g$. Spaces belonging to case (2) admit both Berwaldian Randers metric and also non-Berwaldian Randers metrics of Douglas type. Lie groups of cases (3) and (4) admit only  non-Berwaldian Randers metrics of Douglas type arising from left invariant hyper-Hermitian metrics.
\end{cor}
Now we continue with computing the flag curvature of the Finsler spaces given in theorem \ref{classification theorem}. Flag curvatures of Berwaldian Randers metrics arising from left invariant hyper-Hermitian metrics have been studied in \cite{Salimi2}, so here we give the flag curvature formulas in the case of non-Berwaldian Randers metrics of Douglas type.\\
In \cite{Deng-Hou Canad. J}, the authors have given the following  formula for the flag curvature of left invariant Randers metrics of Douglas type
\begin{equation}\label{flag curvature of Randers metrics of Douglas type}
K^F(P,V)=\dfrac{g^2}{F^2} K^g(P)+\dfrac{1}{4F^4} \left( 3g^2\left( U\left( V,V\right),Q \right)-4Fg\left( U\left( V,U\left( V,V\right) \right),Q  \right)   \right),
\end{equation}
where $K^g$ and $K^F$ denote the sectional curvature of $g$ and the flag curvature of $F$, respectively, and $U:\mathfrak{g} \times \mathfrak{g} \longrightarrow \mathfrak{g}$ is a bilinear symmetric map defined as follows
\begin{equation*}
2g\left( U\left( V,S\right),R \right) = g\left( [R,V],S\right) + g\left( [R,S],V\right), \qquad \forall \ \ R, S, V \in \mathfrak{g}.
\end{equation*}
In our previous work \cite{Hosseini-Salimi1}, the above formula has been simplified to
\begin{equation}\label{Douglas flag curvature}
K^F(P,V)=\dfrac{g^2}{F^2} K^g(P) + \dfrac{1}{4F^4} \left( 3g^2\left( [Q,V],V\right) - 2F \left( g\left( \left[ \left[ Q,V\right],V \right],V \right) - g\left( V,\left[ Q,ad^*_VV\right] \right)   \right) \right),
\end{equation}
where $ad^*_V$ denotes the transpose of $ad_V$ with respect to $g$.\\
Now we compute the flag curvature for any cases.
Let $(P,V)$ be an arbitrary flag at $e$ such that $V=aX+bY+cZ+dW$ for some $a,b,c,d \in \mathbb{R}$.\\

\textbf{Lie groups of the form (i)}.\\
In this case we have $Q=pZ+qW$ for some $p,q \in \mathbb{R}$. A direct computation shows that
\begin{align*}
&g\left( [Q,V],V\right) = -p(a^2+b^2), \\
&g\left( \left[ \left[ Q,V\right],V \right],V \right) = (a^2+b^2) (dq-cp) , \\
&g\left( V,\left[ Q,ad^*_VV\right] \right) = (a^2+b^2) (dq+cp) .
\end{align*}
Hence, using the formula \ref{Douglas flag curvature} we have:
\begin{equation}\label{flag curvature case i}
K^F(P,V)=\dfrac{g^2}{F^2}K^g(P)+\dfrac{1}{4F^4}\left( 3p^2\left( a^2+b^2\right)^2+4cpF\left( a^2+b^2\right)  \right).
\end{equation}

\textbf{Lie groups of the form (ii)}.\\
As it was said previously, in this case we have $Q=qX$ for some $q \in \mathbb{R}$. Similarly, an straightforward computation shows that
\begin{align*}
&g\left( [Q,V],V\right) = q(b^2+c^2+d^2), \\
&g\left( \left[ \left[ Q,V\right],V \right],V \right) =-aq(b^2+c^2+d^2) , \\
&g\left( V,\left[ Q,ad^*_VV\right] \right) =aq(b^2+c^2+d^2) .
\end{align*}
So the flag curvature in this case is as follows:
\begin{equation}\label{flag curvature case ii}
K^F(P,V)=\dfrac{g^2}{F^2}K^g(P)+\dfrac{1}{4F^4}\left( 3q^2\left( b^2+c^2+d^2\right)^2+4aqF\left( b^2+c^2+d^2\right)  \right).
\end{equation}

\textbf{Lie groups of the form (iii)}.\\
As mentioned in theorem \ref{classification theorem}, in this case we have $Q=qX$. Similar to the above we have
\begin{align*}
&g\left( [Q,V],V\right) = q(b^2+\dfrac{1}{2}c^2+\dfrac{1}{2}d^2), \\
&g\left( \left[ \left[ Q,V\right],V \right],V \right) =-aq(b^2+\dfrac{1}{4}c^2+\dfrac{1}{4}d^2) , \\
&g\left( V,\left[ Q,ad^*_VV\right] \right) =aq(b^2+\dfrac{1}{4}c^2+\dfrac{1}{4}d^2).
\end{align*}
Substituting the above equations in the formula \ref{Douglas flag curvature} shows that
\begin{equation}\label{flag curvature case iii}
K^F(P,V)=\dfrac{g^2}{F^2}K^g(P)+\dfrac{1}{4F^4}\left( \dfrac{3}{4}q^2\left(2 b^2+c^2+d^2\right)^2+aqF\left( 4b^2+c^2+d^2\right)  \right).
\end{equation}
The above formulas of flag curvature lead to the following corollary.
\begin{cor}
Let $F$ be a left invariant non-Berwaldian Randers metric of Douglas type on a four-dimensional hypercomplex simply connected Lie group $G$, arising from a left invariant hyper-Hermitian Riemannian metric $g$ and a left invariant vector field $Q$. Then the flag curvature of $F$ and the sectional curvature of $g$ have the same sign in some directions.
\end{cor}
\begin{proof}
It is sufficient to consider $V\in \textsf{span}\{Z,W\}$ for case (i), and $V\in\textsf{span}\{Q\}$ in cases (ii) and (iii).
\end{proof}

{\large{\textbf{Acknowledgment.}}} We are grateful to the office of Graduate Studies of the University of Isfahan for their support. This research was supported by the Center of Excellence for Mathematics at the University of Isfahan.

\end{document}